\documentclass[12pt]{amsart}

\usepackage{amsthm}
\usepackage{amssymb}
\usepackage{amsmath}

\numberwithin{equation}{section}
\newtheorem{thm}{Theorem}[section]
\newtheorem{prop}[thm]{Proposition}
\newtheorem{lem}[thm]{Lemma}
\newtheorem{cor}[thm]{Corollary}

\theoremstyle{definition}

\begin{document}

\title[Restricted sum formula for $q$MZV]
{A restricted sum formula for \\ a $q$-analogue of multiple zeta values}
\author{Yoshihiro Takeyama}
\address{Division of Mathematics, 
Faculty of Pure and Applied Sciences, 
University of Tsukuba, Tsukuba, Ibaraki 305-8571, Japan}
\email{takeyama@math.tsukuba.ac.jp}

\dedicatory{\it Dedicated to Professor Michio Jimbo on his sixtieth birthday} 

\begin{abstract}
We prove a new linear relation for a $q$-analogue of multiple zeta values.  
It is a $q$-extension of the restricted sum formula obtained by 
Eie, Liaw and Ong for multiple zeta values. 
\end{abstract}
\maketitle

\setcounter{section}{0}
\setcounter{equation}{0}


\section{Introduction}

Let $\alpha=(\alpha_{1}, \ldots , \alpha_{r})$ be a multi-index of positive integers.  
We call the values $r$ and $\sum_{i=1}^{r}\alpha_{i}$ 
{\it depth} and {\it weight} of $\alpha$, respectively. 
If $\alpha_{1}\ge 2$, we say that $\alpha$ is {\it admissible}. 
For an admissible index $(\alpha_{1}, \ldots , \alpha_{r})$,  
{\it multiple zeta value} (MZV) is defined by 
\begin{align*}
\zeta(\alpha_{1}, \ldots , \alpha_{r}):=\sum_{m_{1}>\cdots >m_{r}>0}
\frac{1}{m_{1}^{\alpha_{1}} \cdots m_{r}^{\alpha_{r}}}.   
\end{align*} 
Let $I_{0}(r, n)$ be the set of admissible indices of depth $r$ and weight $n$. 
In \cite{ELO}, Eie, Liaw and Ong proved the following relation called  
a restricted sum formula: 
\begin{align}
\sum_{\alpha \in I_{0}(b, n)}\zeta(\alpha_{1}, \ldots , \alpha_{b}, 1^{a})=
\sum_{\beta \in I_{0}(a+1, a+b+1)}\zeta(\beta_{1}+n-b-1, \beta_{2}, \ldots , \beta_{a+1}),   
\label{eq:res-sum-original}
\end{align}
where $a \ge 0, b\ge 1, n\ge b+1$ and $1^{a}$ is an abbreviation of the subsequence
$(1, \ldots , 1)$ of length $a$. 
It is a generalization of the sum formula proved in \cite{Gr, Za}, 
which is the equality \eqref{eq:res-sum-original} with $a=0$.   

In this paper we prove a $q$-analogue of the restricted sum formula. 
Let $0<q<1$. 
For an admissible index $\alpha=(\alpha_{1}, \ldots , \alpha_{r})$, 
a {\it $q$-analogue of multiple zeta value} ($q$MZV) \cite{Bnote, KKW, Z} is defined by 
\begin{align*}
\zeta_{q}(\alpha_{1}, \ldots , \alpha_{r}):=\sum_{m_{1}>\cdots >m_{r}>0}
\frac{q^{(\alpha_{1}-1)m_{1}+\cdots +(\alpha_{r}-1)m_{r}}}
{[m_{1}]^{\alpha_{1}} \cdots [m_{r}]^{\alpha_{r}}},    
\end{align*}
where $[n]$ is the $q$-integer $[n]:=(1-q^{n})/(1-q)$. 
In the limit $q \to 1$, $q$MZV converges to MZV. 
The main theorem of this article claims that 
$q$MZV's also satisfy the restricted sum formula: 

\begin{thm}\label{thm:main}
For any integers $a\ge 0, b\ge 1$ and $n\ge b+1$, it holds that 
\begin{align}
\sum_{\alpha \in I_{0}(b, n)}\zeta_{q}(\alpha_{1}, \ldots , \alpha_{b}, 1^{a})=
\sum_{\beta \in I_{0}(a+1, a+b+1)}\zeta_{q}(\beta_{1}+n-b-1, \beta_{2}, \ldots , \beta_{a+1}).  
\label{eq:main}
\end{align} 
\end{thm}

Setting $a=0$ we recover the sum formula for $q$MZV obtained by Bradley \cite{B}. 

The strategy to prove Theorem \ref{thm:main} is similar to that of 
the proof for MZV's. 
However we should overcome some new difficulties. 
In the calculation of the $q$-analogue case, 
some additional terms are of the form 
$\sum_{n=1}^{\infty}q^{kn}/[n]^{k} \, (k\in\mathbb{Z}_{\ge 1})$.  
In the limit of $q \to 1$, it becomes a harmonic sum $\sum 1/n^{k}$, 
but it is beyond a class of $q$-series described by $q$MZV's. 
To control such terms we make use of algebraic formulation 
of multiple harmonic series given in Section \ref{subsec:algebra}. 
We introduce a noncommutative polynomial algebra $\mathfrak{d}$ 
which is an extension of the algebra used in the proof of 
a $q$-analogue of Kawashima's relation for MZV \cite{qKaw}.  
Then the proof of Theorem \ref{thm:main} is reduced to 
an algebraic calculation in $\mathfrak{d}$ as will be seen in Section \ref{subsec:proof1}.  
We proceed the algebraic computation in Section \ref{subsec:proof2} and 
finish the proof of Theorem \ref{thm:main}.   

Throughout this article we assume that $0<q<1$. 
We denote the set of multi-indices of positive integers, including non-admissible ones, 
of depth $r$ and weight $n$ by $I(r, n)$. 
 
\section{Proof}

\subsection{Summation over indices}

For $b\ge 1, n \ge 2$ and $M \in \mathbb{Z}_{\ge 1}$, define 
\begin{align*}
K_{b, n}(M):=\sum_{\alpha \in I_{0}(b, n)}
\sum_{m_{1}>m_{2}>\cdots >m_{b-1}>m_{b}=M}
\prod_{j=1}^{b}\frac{q^{(\alpha_{j}-1)m_{j}}}{[m_{j}]^{\alpha_{j}}}.  
\end{align*}
Since $\alpha_{1} \ge 2$, the infinite sum in the right hand side is convergent. 
Note that $K_{1, n}(M)=q^{(n-1)M}/[M]^{n}$. 

For positive integers $\ell, \beta, M$ and $N \, (N>M)$, we set 
\begin{align*}
& 
f_{\ell}(N, M):=\sum_{N=k_{1}>k_{2}>\cdots >k_{\ell}>M}\frac{q^{k_{1}-M}}{[k_{1}-M]}
\prod_{j=2}^{\ell}\frac{1}{[k_{j}-M]},  \\
& 
g_{\ell, \beta}(M):=\sum_{M=m_{1}\ge m_{2}\ge \cdots \ge m_{\ell}\ge 1}
\frac{q^{(\beta-1)m_{1}}}{[m_{1}]^{\beta}} 
\prod_{j=2}^{\ell}\frac{q^{m_{j}}}{[m_{j}]}. 
\end{align*}
We set $f_{\ell}(N, M)=0$ unless $N>M$. 
Note that $g_{1, \beta}(M)=K_{1, \beta}(M)$. 

\begin{lem}\label{lem:KKf}
For $M\ge 1$, $b\ge 1$ and $n\ge 2$, it holds that 
\begin{align*}
K_{b, n}(M)=g_{b, n-b+1}(M)-\sum_{s=1}^{b-1}\sum_{N=M+1}^{\infty}K_{b-s, n-s}(N)f_{s}(N, M).  
\end{align*}
\end{lem}

\begin{proof}
For $k\ge 2$ and $m_{1}>m_{2}$, it holds that 
\begin{align*}
&\sum_{\beta \in I(2, k)}
\frac{q^{(\beta_{1}-1)m_{1}+(\beta_{2}-1)m_{2}}}{[m_{1}]^{\beta_{1}}[m_{2}]^{\beta_{2}}} \\
&=\frac{1}{[m_{1}][m_{2}]}\left(
\left(\frac{q^{m_{1}}}{[m_{1}]}\right)^{k-1}-\left(\frac{q^{m_{2}}}{[m_{2}]}\right)^{k-1} 
\right)/\left(\frac{q^{m_{1}}}{[m_{1}]}-\frac{q^{m_{2}}}{[m_{2}]}\right) \\ 
&=\frac{q^{(k-2)m_{2}}}{[m_{2}]^{k-1}}\frac{1}{[m_{1}-m_{2}]}-
\frac{q^{(k-2)m_{1}}}{[m_{1}]^{k-1}}\frac{q^{m_{1}-m_{2}}}{[m_{1}-m_{2}]}. 
\end{align*} 
Using the above formula repeatedly we get 
\begin{align*}
K_{b, n}(M)&=\sum_{m_{1}>\cdots >m_{b-1}>m_{b}=M}
\frac{q^{m_{1}}}{[m_{1}]}
\sum_{\beta \in I(b, n-1)}
\prod_{j=1}^{b}\frac{q^{(\beta_{j}-1)m_{j}}}{[m_{j}]^{\beta_{j}}}
\\ 
&=\sum_{m_{1}>\cdots >m_{b-1}>m_{b}=M}
\frac{q^{m_{1}}}{[m_{1}]}\left(\prod_{j=1}^{b-1}\frac{1}{[m_{j}-m_{b}]}\right)
\frac{q^{(n-b-1)m_{b}}}{[m_{b}]^{n-b}} \\ 
&-\sum_{s=1}^{b-1}\sum_{m_{b-s}=M+1}^{\infty}K_{b-s, n-s}(m_{b-s})f_{s}(m_{b-s}, M).   
\end{align*}
The first term of the right hand side above is rewritten as follows. 
Setting $m_{j}=\ell_{j}+\cdots +\ell_{b-1}+M \, (j=1, \ldots , b-1)$, we have 
\begin{align*}
& 
\sum_{m_{1}>\cdots >m_{b-1}>m_{b}=M}
\frac{q^{m_{1}}}{[m_{1}]}\left(\prod_{j=1}^{b-1}\frac{1}{[m_{j}-m_{b}]}\right) 
\frac{q^{(n-b-1)m_{b}}}{[m_{b}]^{n-b}} \\ 
&\quad {}=\frac{q^{(n-b-1)M}}{[M]^{n-b}}\sum_{\ell_{1}, \ldots , \ell_{b-1}=1}^{\infty}
\frac{q^{\ell_{1}+\cdots +\ell_{b-1}+M}}{[\ell_{1}+\cdots +\ell_{b-1}+M]}
\prod_{j=1}^{b-1}\frac{1}{[\ell_{j}+\cdots +\ell_{b-1}]}.
\end{align*}
Now take the sum with respect to $\ell_{1}, \ell_{2}, \ldots , \ell_{b-1}$ successively 
using the equality 
\begin{align*}
& 
\sum_{\ell=1}^{\infty}\frac{q^{\ell+m}}{[\ell+m]}\frac{1}{[\ell+n]}=
\sum_{\ell=1}^{\infty}
\left(\frac{q^{\ell+n}}{[\ell+n]}-\frac{q^{\ell+m}}{[\ell+m]}\right)
\frac{q^{m-n}}{[m-n]}=
\frac{q^{m-n}}{[m-n]}\sum_{\ell=1}^{m-n}\frac{q^{\ell+n}}{[\ell+n]}
\end{align*}
which holds for any $m>n$.  
Then we obtain $g_{b, n-b+1}(M)$. 
\end{proof}

Lemma \ref{lem:KKf} implies the following proposition, 
which can be proved by induction on $b$: 

\begin{prop}\label{prop:K}
For positive integers $r, \ell$ and 
$N_{1}>\cdots >N_{r}>M$, set 
\begin{align}
h_{r, \ell}(N_{1}, \ldots , N_{r}, M):=
\sum_{c \in I(r, \ell)}
\left(\prod_{j=1}^{r-1}f_{c_{j}}(N_{j}, N_{j+1})\right)
f_{c_{r}}(N_{r}, M).
\label{eq:def-hrs} 
\end{align}
Then 
\begin{align}
&K_{b, m}(M)=g_{b, m-b+1}(M) \label{eq:propK} \\ 
&+\sum_{\ell=1}^{b-1}\sum_{r=1}^{\ell}(-1)^{r} 
\sum_{N_{1}>N_{2}>\cdots >N_{r}>M}\!\!\!
g_{b-\ell, m-b+1}(N_{1})\, 
h_{r, \ell}(N_{1}, \ldots , N_{r}, M)
\nonumber 
\end{align} 
for $b \ge 1, m \ge 2$ and $M\ge 1$. 
\end{prop}

Multiply $K_{b, m}(M)$ by the harmonic sum 
\begin{align}
\sum_{M>m_{1}>\cdots >m_{a}>0}\prod_{j=1}^{a}\frac{1}{[m_{j}]}
\label{eq:harmonic-tail}
\end{align}
and take the sum over all $M\ge 1$. 
Then we get the left hand side of \eqref{eq:main}. 
In order to carry out the same calculation for the right hand side of \eqref{eq:propK},  
we prepare an algebraic formulation for multiple harmonic sums.

\subsection{Algebraic structure of multiple harmonic sums}\label{subsec:algebra}

Denote by $\mathfrak{d}$ the non-commutative polynomial algebra over $\mathbb{Z}$ 
freely generated by the set of alphabets 
$S=\{z_{k}\}_{k=1}^{\infty} \cup \{\xi_{k}\}_{k=1}^{\infty}$. 
For a positive integer $m$, set 
\begin{align*}
J_{z_{k}}(m):=\frac{q^{(k-1)m}}{[m]^k}, \qquad 
J_{\xi_{k}}(m):=\frac{q^{km}}{[m]^{k}}.  
\end{align*}
For a word $w=u_{1} \cdots u_{r}\in \mathfrak{d} \, (r \ge 1, u_{i}\in S)$ and 
$M \in \mathbb{Z}_{\ge 1}$, set 
\begin{align*}
A_{w}(M)&:=\sum_{M>m_{1}>\cdots >m_{r}>0}
J_{u_{1}}(m_{1}) \cdots J_{u_{r}}(m_{r}), \\     
A_{w}^{\star}(M)&:=\sum_{M>m_{1}\ge\cdots \ge m_{r}\ge 1}
J_{u_{1}}(m_{1}) \cdots J_{u_{r}}(m_{r}). 
\end{align*}
We extend the maps $w \mapsto A_{w}(M)$ and $w \mapsto A_{w}^{\star}(M)$ to 
the $\mathbb{Z}$-module homomorphisms $A(M), A^{\star}(M): \mathfrak{d} \to \mathbb{R}$ 
by $A_{1}(M)=1, A^{\star}_{1}(M)=1$ and $\mathbb{Z}$-linearity. 
Note that $A_{z_{1}^{a}}(M)$ is equal to the harmonic sum \eqref{eq:harmonic-tail}. 
If $w$ is contained in the $\mathbb{Z}$-linear span of monomials $z_{i_{1}} \cdots z_{i_{r}}$ 
with $i_{1} \ge 2$, 
$A_{w}(M)$ becomes a linear combination of $q$MZV's in the limit $M \to \infty$.      

Denote by $\mathfrak{d}_{\xi}$ the $\mathbb{Z}$-subalgebra of $\mathfrak{d}$ 
generated by $\{\xi_{k}\}_{k=1}^{\infty}$. 
Define a $\mathbb{Z}$-bilinear map 
$\rho : \mathfrak{d}_{\xi} \times \mathfrak{d} \to \mathfrak{d}$ 
inductively by 
$\rho(1, w)=w \, (w \in \mathfrak{d}), \, \rho(v, 1)=v \, (v \in \mathfrak{d}_{\xi})$ and 
\begin{align*}
& 
\rho(\xi_{k}v, z_{\ell}w)=\xi_{k}\,\rho(v, z_{\ell}w)+z_{\ell}\,\rho(\xi_{k}v, w)+z_{k+\ell}\,\rho(v,w), \\ 
& 
\rho(\xi_{k}v, \xi_{\ell}w)=\xi_{k}\,\rho(v, z_{\ell}w)+\xi_{\ell}\,\rho(\xi_{k}v,w)+\xi_{k+\ell}\,\rho(v,w) 
\end{align*}
for $v \in \mathfrak{d}_{\xi}$ and $w \in \mathfrak{d}$. 

\begin{prop}\label{prop:A-prod}
For $v \in \mathfrak{d}_{\xi}, w \in \mathfrak{d}$ and $M\ge 1$, 
we have $A_{v}(M)A_{w}(M)=A_{\rho(v,w)}(M)$. 
\end{prop}

\begin{proof}
It is enough to consider the case where $v$ and $w$ are words. 
If $v=1$ or $w=1$, it is trivial. 
{}From the definition of $A(M)$, it holds that 
\begin{align}
A_{\xi_{k}w}(M)=\sum_{M>m>0}\frac{q^{km}}{[m]^{k}}A_{w}(m), \quad 
A_{z_{\ell}w}(M)=\sum_{M>n>0}\frac{q^{(\ell-1)n}}{[n]^{\ell}}A_{w}(n). 
\label{eq:A-recursion}
\end{align}
Hence we find 
\begin{align*}
A_{\xi_{k}v}(M)A_{z_{\ell}w}(M)&=\left(
\sum_{M>m>n>0}+\sum_{M>n>m>0}+\sum_{M>m=n>0}\right)
\frac{q^{km}}{[m]^{k}}\frac{q^{(\ell-1)n}}{[n]^{\ell}}A_{v}(m)A_{w}(n) \\ 
&=\sum_{M>m>0}\frac{q^{km}}{[m]^{k}}A_{v}(m)A_{z_{\ell}w}(m)+
\sum_{M>n>0}\frac{q^{(\ell-1)n}}{[n]^{\ell}}A_{\xi_{k}v}(n)A_{w}(n) \\ 
&+\sum_{M>m>0}\frac{q^{(k+\ell-1)m}}{[m]^{k+\ell}}A_{v}(m)A_{w}(m)   
\end{align*}
and a similar formula for $A_{\xi_{k}v}(M)A_{\xi_{\ell}w}(M)$. 
Now the proposition follows from the induction on 
the sum of length of $v$ and $w$. 
\end{proof}

For $k \ge 1$, we define a $\mathbb{Z}$-linear map 
$\xi_{k}\circ \cdot \, : \mathfrak{d}_{\xi} \to \mathfrak{d}_{\xi}$ 
inductively by 
$\xi_{k}\circ 1=0$ and $\xi_{k}\circ (\xi_{\ell}v)=\xi_{k+\ell}v$ 
for $v \in \mathfrak{d}_{\xi}$.  
Now consider the $\mathbb{Z}$-linear map $d:\mathfrak{d}_{\xi} \to \mathfrak{d}_{\xi}$ 
defined by $d(1)=1$ and 
$d(\xi_{k}v)=\xi_{k}d(v)+\xi_{k}\circ d(v) \, (v \in \mathfrak{d}_{\xi})$. 

\begin{prop}\label{prop:star-to-A}
For any $v \in \mathfrak{d}_{\xi}$ and $M\ge 1$, 
it holds that $A_{v}^{\star}(M)=A_{d(v)}(M)$.  
\end{prop}

\begin{proof}
{}From the definition of $A(M)$ and $A^{\star}(M)$ we have 
\begin{align*}
A_{\xi_{k}v}^{\star}(M)=\sum_{M>m>0}\frac{q^{km}}{[m]^{k}}A^{\star}_{v}(m+1), 
\quad 
\sum_{M>m>0}\frac{q^{km}}{[m]^{k}}A_{v}(m+1)=A_{\xi_{k}v+\xi_{k}\circ v}(M). 
\end{align*}
To show the second formula, divide the sum $A_{v}(m+1)$ into the two parts 
with $m_{1}=m$ and $m_{1}<m$. 
Combining the two formulas above, we obtain the proposition 
by induction on length of $v$. 
\end{proof}

\subsection{Algebraic formulation of the main theorem}\label{subsec:proof1}

To calculate the right hand side of \eqref{eq:propK} multiplied by 
the harmonic sum \eqref{eq:harmonic-tail}, 
we need the following formula: 

\begin{lem}\label{lem:f-through}
For $n_{1}>\cdots >n_{s}>n_{s+1}>0$, set  
\begin{align}
p(n_{1}, \ldots , n_{s}; n_{s+1}):=\frac{q^{n_{1}-n_{s+1}}}{[n_{1}-n_{s+1}]}
\prod_{j=2}^{s}\frac{1}{[n_{j}-n_{s+1}]}.  
\label{eq:def-p}
\end{align}
Let $s\ge 1$, $v=z_{1}$ or $\xi_{1}$, and $N$ and $M$ be positive integers such that $N>M$. 
Then it holds that  
\begin{align*} 
& 
\sum_{N>n_{1}>\cdots >n_{s+1}>M}p(n_{1}, \ldots , n_{s}; n_{s+1})J_{v}(n_{s+1}) \\ 
&=\sum_{N>k_{1}>\cdots >k_{s+1}>M}
J_{v}(k_{1})p(k_{2}, \ldots , k_{s}, k_{s+1}; M) \\ 
&+\sum_{i=1}^{s}\sum_{N>k_{1}>\cdots >k_{s+1}>M}\frac{q^{k_{1}}}{[k_{1}]}
\left( \prod_{j=2}^{i+1}\frac{1}{[k_{j}]} \right)
p(k_{i+2}, \ldots , k_{s+1}; M), 
\end{align*}
where $p(\emptyset; M)=1$.  
\end{lem}

\begin{proof}
Here we prove the lemma in the case of $v=z_{1}$. 
The proof for $v=\xi_{1}$ is similar. 
Using 
\begin{align*}
& 
\frac{1}{[n_{1}-n_{s+1}][n_{s+1}]}=\frac{1}{[n_{1}-n_{s+1}]}+\frac{q^{n_{s+1}}}{[n_{s+1}]}, \\
& 
\frac{1}{[n_{j}-n_{s+1}][n_{s+1}]}=\frac{q^{n_{j}-n_{s+1}}}{[n_{j}-n_{s+1}]}+\frac{1}{[n_{s+1}]} 
\quad (j=2, \ldots , s),   
\end{align*}
we find that 
\begin{align*}
& 
p(n_{1}, \ldots , n_{s}; n_{s+1})J_{v}(n_{s+1}) \\ 
&=\sum_{i=0}^{s}\frac{q^{(1-\delta_{i,0})n_{1}}}{[n_{1}]}
\left( \prod_{j=2}^{i+1}\frac{1}{[n_{j}]} \right)
p(n_{i+1}, \ldots , n_{s}; n_{s+1}). 
\end{align*}
Now take the sum of the both hand sides over $N>n_{1}>\cdots >n_{s+1}>M$. 
In the right hand side, change the variables 
$n_{1}, \ldots , n_{s+1}$ to $k_{1}, \ldots , k_{s+1}$ by setting  
$n_{t}=k_{t} \, (1 \le t \le i+1), \, n_{t}=k_{i+1}-k_{i+2}+k_{t+1} \, (i+2 \le t \le s)$ 
and $n_{s+1}=k_{i+1}-k_{i+2}+M$. 
Then we get the desired formula. 
\end{proof}

Let $\mathfrak{d}_{1}$ be the $\mathbb{Z}$-subalgebra of $\mathfrak{d}$ 
generated by $z_{1}$ and $\xi_{1}$. 
Motivated by Lemma \ref{lem:f-through} we introduce the $\mathbb{Z}$-module homomorphism 
$\varphi_{s}\,:\, \mathfrak{d}_{1} \to \mathfrak{d}_{1} \, (s \in \mathbb{Z}_{\ge 0})$ 
defined in the following way. 
Determine $\varphi_{s}(w)$ for a word $w \in \mathfrak{d}_{1}$ inductively 
on $s$ and length of $w$ by 
$\varphi_{0}={\rm id}, \, \varphi_{s}(1)=\xi_{1}z_{1}^{s-1}\, (s \ge 1)$ and 
\begin{align*}
\varphi_{s}(z_{1}w)=z_{1}\varphi_{s}(w)+\xi_{1}\sum_{i=1}^{s}z_{1}^{i}\varphi_{s-i}(w), \qquad 
\varphi_{s}(\xi_{1}w)=\xi_{1}\sum_{i=0}^{s}z_{1}^{i}\varphi_{s-i}(w),  
\end{align*}
and extend it by $\mathbb{Z}$-linearity.

\begin{prop}\label{prop:contraction}
For $w \in \mathfrak{d}_{1}$ and any positive integers $s, s', \ell, \beta$ and $N$, we have 
\begin{align}
\sum_{N>M_{1}>M_{2}>0}f_{s'}(N, M_{1})f_{s}(M_{1}, M_{2})A_{w}(M_{2})
&=\sum_{N>M>0}f_{s'}(N, M)A_{\varphi_{s}(w)}(M), 
\label{eq:ffA} \\ 
\sum_{M_{1}>M_{2}>0}g_{\ell, \beta}(M_{1})f_{s}(M_{1}, M_{2})A_{w}(M_{2})
&=\sum_{M>0}g_{\ell,\beta}(M)A_{\varphi_{s}(w)}(M). 
\nonumber 
\end{align} 
\end{prop} 

\begin{proof}
Here we prove the first formula \eqref{eq:ffA}. 
The proof for the second is similar. 
It suffices to consider the case where 
$w=u_{1}\cdots u_{r} \, (r \ge 1, u_{i}\in S)$ is a word. 
The left hand side of \eqref{eq:ffA} is equal to 
\begin{align*}
\sum f_{s'}(N, M_{1})\,
p(M_{1}, k_{1}, \ldots , k_{s-1}; M_{2})
\prod_{i=1}^{r}J_{u_i}(m_{i}), 
\end{align*}
where $p$ is defined by \eqref{eq:def-p} and 
the sum is over $M_{1}, k_{i} \, (1\le i \le s-1), M_{2}, m_{i} \, (1 \le i \le r)$ 
with the condition 
$N>M_{1}>k_{1}>\cdots >k_{s-1}>M_{2}>m_{1}>\cdots >m_{r}>0$.  
Changing the variables $(k_{1}, \ldots , k_{s-1}, M_{2})$ to 
$(n_{1}, \ldots , n_{s})$ by 
$k_{i}=M_{1}-n_{1}+n_{i+1} \, (1 \le i \le s-1)$ and $M_{2}=M_{1}-n_{1}+m_{1}$, 
we obtain 
\begin{align*}
\sum f_{s'}(N, M_{1})p(n_{1}, \ldots , n_{s}; m_{1})
\prod_{i=1}^{r}J_{u_{i}}(m_{i}), 
\end{align*}
where the sum is over  
$N>M_{1}>n_{1}>\cdots >n_{s}>m_{1}>\cdots >m_{r}>0$.  
{}From Lemma \ref{lem:f-through} and the definition of $\varphi_{s}$,   
we see by induction on $r$ that it is equal to the right hand side of \eqref{eq:ffA}.  
\end{proof}

We define the $\mathbb{Z}$-linear maps  
$\Phi_{\ell}: \mathfrak{d}_{1} \to \mathfrak{d}_{1} \, (\ell \ge 0)$ by 
$\Phi_{0}:={\rm id}$ and 
\begin{align*}
\Phi_{\ell}:=\sum_{r=1}^{\ell}(-1)^{r}\!\!\!\sum_{c \in I(r, \ell)}
\varphi_{c_{1}} \cdots \varphi_{c_{r}},     
\end{align*}
and $Z_{s} : \mathfrak{d}_{1}\to \mathfrak{d} \, (s\ge 0)$ by 
\begin{align*}
Z_{s}(w):=\sum_{\ell=0}^{s}\rho(d(\xi_{1}^{s-\ell}), \Phi_{\ell}(w)).  
\end{align*}

\begin{prop}\label{prop:main-0}
For any integers $a\ge 0, b\ge 1$ and $n\ge b+1$, we have 
\begin{align}
\sum_{\alpha \in I_{0}(b, n)}\zeta_{q}(\alpha_{1}, \ldots , \alpha_{b}, 1^{a})=
\sum_{s=0}^{b-1}\sum_{M>0}\frac{q^{(n-s-1)M}}{[M]^{n-s}}
A_{Z_{s}(z_{1}^{a})}(M).  
\label{eq:prop-main-0}
\end{align}  
\end{prop}

\begin{proof}
Using Proposition \ref{prop:contraction} repeatedly, we have 
\begin{align*}
& 
\sum_{N_{1}>N_{2}>\cdots >N_{r}>M>0}\!\!\!
g_{b, m}(N_{1})\, 
h_{r, \ell}(N_{1}, \ldots , N_{r}, M)A_{z_{1}^{a}}(M) \\ 
&=\sum_{c \in I(r, \ell)}\sum_{M>0}g_{b,m}(M)
A_{\varphi_{c_{1}} \cdots \varphi_{c_{r}}(z_{1}^{a})}(M),   
\end{align*}
where $h_{r,\ell}$ is defined by \eqref{eq:def-hrs}.  
Hence Proposition \ref{prop:K} implies that 
\begin{align*}
\sum_{\alpha \in I_{0}(b, m)}\zeta_{q}(\alpha_{1}, \ldots , \alpha_{b}, 1^{a})&=
\sum_{M>0}K_{b, m}(M)A_{z_{1}^{a}}(M) \\ 
&=\sum_{\ell=0}^{b-1}\sum_{M>0}
g_{b-\ell, m-b+1}(M)A_{\Phi_{\ell}(z_{1}^{a})}(M).
\end{align*} 
Substituting 
\begin{align*}
g_{j, n}(M)=\sum_{t=0}^{j-1}\frac{q^{(n+j-t-2)M}}{[M]^{n+j-t-1}}A^{\star}_{\xi_{1}^{t}}(M),  
\end{align*}
we get the desired formula from 
Proposition \ref{prop:A-prod} and Proposition \ref{prop:star-to-A}. 
\end{proof}

As we will see in the next subsection, 
the elements $Z_{s}(z_{1}^{a}) \, (s, a \ge 0)$ belong to 
the subalgebra of $\mathfrak{d}$ generated only by $\{z_{k}\}_{k=1}^{\infty}$ 
(see Proposition \ref{prop:main-2} and Proposition \ref{prop:main-3} below). 
Thus the right hand side of \eqref{eq:prop-main-0} will turn out to be a linear combination of 
$q$MZV's. 

\subsection{Proof of the main theorem}\label{subsec:proof2}

First we give a proof of Theorem \ref{thm:main} with $a=0$, 
that is, the sum formula for $q$MZV's. 
To this aim we prepare a recurrence relation of $d(\xi_{1}^{k}) \, (k \ge 0)$. 

\begin{lem}
Let $k\ge 1$. Then 
\begin{align*}
d(\xi_{1}^{k})=\sum_{r=1}^{k}
\sum_{c \in I(r, k)}\xi_{c_{1}} \cdots \xi_{c_{r}}.  
\end{align*} 
\end{lem}

\begin{proof}
We prove the lemma by induction on $k$.  
The case of $k=1$ is trivial. Let $k \ge 2$. 
{}From the definition of $d$ and the hypothesis of induction we see that 
\begin{align*}
d(\xi_{1}^{k})&=d(\xi_{1}\cdot \xi_{1}^{k-1})=\xi_{1}
\sum_{r=1}^{k-1}\sum_{c \in I(r, k-1)}\xi_{c_{1}} \cdots \xi_{c_{r}}+\xi_{1}\circ
\left(\sum_{r=1}^{k-1}\sum_{c \in I(r, k-1)}\xi_{c_{1}} \cdots \xi_{c_{r}}\right) \\ 
&=\sum_{r=2}^{k}\sum_{c \in I(r, k) \atop c_{1}=1}\xi_{c_{1}} \cdots \xi_{c_{r}}+
\sum_{r=1}^{k-1}\sum_{c \in I(r, k) \atop c_{1}\ge 2}\xi_{c_{1}} \cdots \xi_{c_{r}}=
\sum_{r=1}^{k}\sum_{c \in I(r, k)}\xi_{c_{1}} \cdots \xi_{c_{r}}. 
\end{align*}
This completes the proof. 
\end{proof}

\begin{cor}
For $k \ge 1$ it holds that 
\begin{align}
d(\xi_{1}^{k})=\sum_{a=1}^{k}\xi_{a}d(\xi_{1}^{k-a}).   
\label{eq:d-recursion}
\end{align}
\end{cor}

The sum formula for $q$MZV's follows from the following proposition. 

\begin{prop}\label{prop:main-2}
$Z_{s}(1)=\delta_{s, 0} \, (s \ge 0)$.  
\end{prop}

\begin{proof}
Using $\Phi_{\ell}=-\sum_{a=1}^{\ell}\varphi_{a}\Phi_{\ell-a} \, (\ell \ge 1)$, 
we find that $\Phi_{\ell}(1)=(-\xi_{1})^{\ell} \, (\ell\ge 0)$ 
by induction on $\ell$.  
Thus the proposition is reduced to the proof of 
\begin{align*}
\sum_{\ell=0}^{s}(-1)^{\ell}\rho(d(\xi_{1}^{s-\ell}), \xi_{1}^{\ell})=\delta_{s,0}.  
\end{align*}
Let us prove it by induction on $s$. 
Denote the left hand side above by $T_{s}$. 
It is trivial that $T_{0}=1$. 
Let $s \ge 1$. Divide $T_{s}$ into the three parts 
\begin{align*}
T_{s}=d(\xi_{1}^{s})+\sum_{\ell=1}^{s-1}(-1)^{\ell}\rho(d(\xi_{1}^{s-\ell}), \xi_{1}^{\ell})+
(-1)^{s}\xi_{1}^{s}. 
\end{align*}
Rewrite the second part by using 
\eqref{eq:d-recursion} and the definition of $\rho$ and $d$. 
Then we get 
\begin{align}
\sum_{a=1}^{s-1}\xi_{a} \sum_{\ell=1}^{s-a}(-1)^{\ell}\rho(d(\xi_{1}^{s-a-\ell}), \xi_{1}^{\ell})-
\sum_{\ell=0}^{s-2}(-1)^{\ell}\xi_{1}\rho(d(\xi_{1}^{s-1-\ell}), \xi_{1}^{\ell})-\sum_{a=1}^{s-1}
\xi_{a+1}I_{s-a-1}. 
\label{eq:decomp-Ts}
\end{align}
{}From $(-1)^{s}\xi_{1}^{s}=-(-1)^{s-1}\xi_{1}\rho(d(\xi_{1}^{0}), \xi_{1}^{s-1})$, 
which is the summand of the second term of \eqref{eq:decomp-Ts} with $\ell=s-1$, and 
\begin{align*}
d(\xi_{1}^{s})=\sum_{a=1}^{s-1}\xi_{a}\rho(d(\xi_{1}^{s-a}), \xi_{1}^{0})+\xi_{s}, 
\end{align*}
we obtain 
\begin{align*}
T_{s}=\sum_{a=1}^{s-1}\xi_{a}T_{s-a}+\xi_{s}-\xi_{1}T_{s-1}-\sum_{a=1}^{s-1}\xi_{a+1}T_{s-a-1}.   
\end{align*}
Therefore the induction hypothesis $T_{a}=\delta_{a, 0} \, (a<s)$ implies that $T_{s}=0$. 
\end{proof}

{}From Proposition \ref{prop:main-0} with $a=0$ and 
Proposition \ref{prop:main-2}, we see that 
\begin{align*}
\sum_{\alpha \in I_{0}(b, n)}\zeta_{q}(\alpha_{1}, \ldots , \alpha_{b})=
\sum_{M>0}\frac{q^{(n-1)M}}{[M]^{n}}A_{1}(M)=\zeta_{q}(n).  
\end{align*}
Thus we get Theorem \ref{thm:main} in the case of $a=0$. 
To complete the proof of Theorem \ref{thm:main}, 
we should calculate $Z_{s}(z_{1}^{a})$ for $a \ge 1$. 
For that purpose we prepare several lemmas. 

\begin{lem}
For $\ell \ge 0$ and $w \in \mathfrak{d}_{1}$, it holds that 
\begin{align}
\Phi_{\ell}(z_{1}w)=\sum_{j=0}^{\ell}(-\xi_{1})^{\ell-j}z_{1}\Phi_{j}(w).  
\label{eq:Phi-recursion}
\end{align} 
\end{lem}

\begin{proof}
For non-negative integers $a$ and $n$, set $\eta_{0, n}=\delta_{n, 0}$ and 
\begin{align*}
\eta_{a, n}:=\sum_{c \in I(a, n)}
\xi_{1}z_{1}^{c_{1}} \cdots \xi_{1}z_{1}^{c_{a}} \quad (a \ge 1). 
\end{align*}
Then it holds that 
\begin{align*}
\varphi_{s}(\xi_{1}^{a}z_{1}w)=\sum_{t=0}^{s}(\eta_{a, s-t}+\eta_{a+1, s-t-1})z_{1}\varphi_{t}(w),  
\end{align*}
where $\eta_{a+1, -1}:=0$, for $a\ge 0, s \ge 0$ and $w \in \mathfrak{d}_{1}$. 
Using this formula we prove \eqref{eq:Phi-recursion} by induction on $\ell$. 
The case of $\ell=0$ is trivial. 
Let $\ell \ge 1$. 
The induction hypothesis and the relation $\Phi_{\ell}=-\sum_{a=1}^{\ell}\varphi_{a}\Phi_{\ell-a}$ 
imply that 
\begin{align*}
\Phi_{\ell}(z_{1}w)=\sum_{j=0}^{\ell-1}\sum_{a=1}^{\ell-j}\sum_{t=0}^{a}
(\eta_{\ell-a-j, a-t}+\eta_{\ell-a-j+1, a-t-1})z_{1}\varphi_{t}(\Phi_{j}(w)).  
\end{align*}
Divide the sum into the two parts with $t=0$ and $t \ge 1$, 
and take the sum with respect to $a$. 
Then we obtain 
\begin{align*}
\sum_{j=0}^{\ell-1}\left\{ 
(-\delta_{\ell-j,0}+(-1)^{\ell-j}\eta_{\ell-j,0})z_{1}\varphi_{0}(\Phi_{j}(w))-\sum_{t=1}^{\ell-j}
\delta_{\ell-j-t,0}z_{1}\varphi_{t}(\Phi_{j}(w))
\right\}.  
\end{align*}
Since $\eta_{\ell-j,0}=\xi_{1}^{\ell-j}$, $\varphi_{0}={\rm id}$ and 
$-\sum_{j=0}^{\ell-1}\varphi_{\ell-j}\Phi_{j}=\varphi_{\ell}$, 
it is equal to the right hand side of \eqref{eq:Phi-recursion}. 
\end{proof}

\begin{lem}\label{lem:rho-to-z}
For $k \ge 0$ and $w \in \mathfrak{d}_{1}$, it holds that 
\begin{align}
\sum_{\ell=0}^{k}\rho(d(\xi_{1}^{k-\ell}), \xi_{1}^{\ell}z_{1}w)=\sum_{\ell=0}^{k}
z_{\ell+1}\,\rho(d(\xi_{1}^{k-\ell}), w).  
\label{eq:rho-to-z}
\end{align} 
\end{lem}

\begin{proof}
Denote the left hand side and the right hand side of \eqref{eq:rho-to-z} by $L_{k}$ and $R_{k}$, 
respectively.  
The equality \eqref{eq:rho-to-z} holds when $k=0$ 
because $L_{0}=\rho(1, z_{1}w)=z_{1}w=z_{1}\rho(1, w)=R_{0}$. 
Hereafter we assume that $k \ge 1$. 

Divide $L_{k}$ into the three parts 
\begin{align}
L_{k}=\rho(d(\xi_{1}^{k}), z_{1}w)+\sum_{\ell=1}^{k-1}(-1)^{\ell}
\rho(d(\xi_{1}^{k-\ell}), \xi_{1}^{\ell}z_{1}w)+
(-1)^{k}\xi_{1}^{k}z_{1}w.  
\label{eq:lem-I-divide}
\end{align}
Let us rewrite the first part. 
Substitute \eqref{eq:d-recursion} into $d(\xi_{1}^{k})$. 
{}From the definition of $\rho$ we see that the first part is equal to 
\begin{align*}
\sum_{a=1}^{k}\left( 
\xi_{a}\rho(d(\xi_{1}^{k-a}), z_{1}w)+z_{1}\rho(\xi_{a}d(\xi_{1}^{k-a}), w)+z_{a+1}\rho(d(\xi_{1}^{k-a}), w)
\right). 
\end{align*}
Note that the first term of the summand with $a=k$ is equal to $\xi_{k}z_{1}w=\xi_{k}L_{0}$. 
Apply \eqref{eq:d-recursion} again to the second term, and we see that 
the first part of the right hand side of \eqref{eq:lem-I-divide} is equal to  
\begin{align}
\xi_{k}L_{0}+\sum_{a=1}^{k-1}\xi_{a}\rho(d(\xi_{1}^{k-a}), z_{1}w)+R_{k}. 
\label{eq:I-part1}
\end{align}

We proceed the same calculation for the second part of \eqref{eq:lem-I-divide}. 
Here we decompose $\xi_{1}^{\ell}z_{1}w=\xi_{1} \cdot \xi_{1}^{\ell-1}z_{1}w$ and 
use \eqref{eq:d-recursion}.  
As a result we get 
\begin{align}
& 
\sum_{a=1}^{k-1}\sum_{\ell=1}^{k-a}(-1)^{\ell}\xi_{a}\rho(d(\xi_{1}^{k-\ell-a}), \xi_{1}^{\ell}z_{1}w)-
\sum_{a=1}^{k-1}\xi_{a+1}I_{k-a-1} 
\label{eq:I-part2} 
\\ 
&-\sum_{\ell=0}^{k-2}(-1)^{l}\xi_{1}\rho(d(\xi_{1}^{k-1-\ell}), \xi_{1}^{\ell}z_{1}w). 
\nonumber 
\end{align} 

Note that the third part of \eqref{eq:lem-I-divide} is equal to 
\begin{align}
{}-(-1)^{k-1}\xi_{1}\rho(d(\xi_{1}^{0}), \xi_{1}^{k-1}z_{1}w),   
\label{eq:I-part3} 
\end{align}
which is the summand of the third term of \eqref{eq:I-part2} with $\ell=k-1$. 
Hence the three parts \eqref{eq:I-part1}, \eqref{eq:I-part2} and \eqref{eq:I-part3} add up to 
\begin{align*}
\xi_{k}L_{0}+\sum_{a=1}^{k-1}\xi_{a}L_{k-a}+R_{k}-\sum_{a=1}^{k-1}\xi_{a+1}L_{k-a-1}-\xi_{1}L_{k-1}=R_{k}.  
\end{align*}
This completes the proof. 
\end{proof}

Now we can prove the key formula to calculate $Z_{s}(z_{1}^{a})$ for $a \ge 1$: 

\begin{prop}\label{prop:main-1}
Let $w \in \mathfrak{d}_{1}$ and $s \ge 0$. Then 
$Z_{s}(z_{1}w)=\sum_{\ell=0}^{s}z_{\ell+1}Z_{s-\ell}(w)$.  
\end{prop}

\begin{proof}
Using \eqref{eq:Phi-recursion} we have 
\begin{align*}
Z_{s}(z_{1}w)=\sum_{\ell=0}^{s}\sum_{j=0}^{l}(-1)^{\ell-j}
\rho(d(\xi_{1}^{s-\ell}), \xi_{1}^{\ell-j}z_{1}\Phi_{j}(w)). 
\end{align*} 
Because of Lemma \ref{lem:rho-to-z} it is equal to 
\begin{align*}
\sum_{j=0}^{s}\sum_{\ell=0}^{s-j}z_{\ell+1}\rho(d(\xi_{1}^{s-j-\ell}), \Phi_{j}(w))=
\sum_{\ell=0}^{s}z_{\ell+1}Z_{s-l}(w).  
\end{align*}
This completes the proof. 
\end{proof}

Combining Proposition \ref{prop:main-2} and Proposition \ref{prop:main-1}, 
we obtain the following formula: 

\begin{prop}\label{prop:main-3}
For $s \ge 0$ and $a \ge 1$, it holds that 
\begin{align*}
Z_{s}(z_{1}^{a})=\sum_{\gamma \in I(a, s+a)}
z_{\gamma_{1}} \cdots z_{\gamma_{a}}. 
\end{align*}
\end{prop}

At last let us prove Theorem \ref{thm:main} in the case of $a\ge 1$. 
{}From Proposition \ref{prop:main-0} and Proposition \ref{prop:main-3}, 
it holds that 
\begin{align*}
\sum_{\alpha \in I_{0}(b, n)}\zeta_{q}(\alpha_{1}, \ldots , \alpha_{b}, 1^{a})=
\sum_{s=0}^{b-1}\sum_{\gamma \in I(a, s+a)}
\zeta_{q}(n-s-1, \gamma_{1}, \ldots , \gamma_{a}). 
\end{align*}
Set $\beta_{1}=b+1-s$. The right hand side becomes 
\begin{align*}
& 
\sum_{\beta_{1}=2}^{b+1}\sum_{\gamma \in I(a, a+b+1-\beta_{1})}
\zeta_{q}(\beta_{1}+n-b-1, \gamma_{1}, \ldots , \gamma_{a}) \\ 
&=\sum_{\beta \in I_{0}(a+1, a+b+1)}
\zeta_{q}(\beta_{1}+n-b-1, \beta_{2}, \ldots , \beta_{a+1}). 
\end{align*} 
This completes the proof of Theorem \ref{thm:main}. 

\medskip 
\section*{Acknowledgments}

The research of the author is supported by Grant-in-Aid for 
Young Scientists (B) No.\,23740119. 
The author is grateful to Yasuo Ohno for helpful informations.

\end{document}